\documentclass[11pt]{article}
\usepackage{cite}
\usepackage{enumitem}
 \usepackage{amsfonts}
\usepackage{amsmath}
\usepackage{amsthm}
\usepackage{amssymb}
\DeclareMathOperator{\cl}{cl}
\usepackage{amscd}
\newtheorem{definition}{Definition}[section]
\newtheorem{theorem}{Theorem}[section]
\newtheorem{corollary}{Corollary}[section]
\newtheorem{lemma}{Lemma}[section]

\newtheorem{remark}{Remark}[section]
\newtheorem{example}{Example}[section]

\begin{document}
\title{{Global and local error bounds: characterizations via directional derivatives and tangent cones}}
\author{{Yu Han \footnote{Corresponding author,  E-mail: hanyumath@163.com} }\\
{\small\it School of Statistics and Data Science, Jiangxi University of Finance and Economics, }\\
{\small\it  Nanchang, Jiangxi 330013,  China} }
\date{}
\maketitle
\vspace*{-9mm}
\begin{center}
\begin{minipage}{6.2in}
{\bf Abstract.} We develop new characterizations of both global and local error bounds for general functions, using directional derivatives and tangent cones without imposing convexity or linear structure.  We first establish several equivalent conditions for the global error bound of a nonnegative lower semicontinuous function.   These equivalences hold for general, possibly nonconvex and nonsmooth functions.  We further link the error bound with perturbation stability, Hausdorff stability of sublevel sets, and an inverse-sublevel-set estimate.  Turning to directional derivatives, we introduce the minimal unit-sphere directional derivative \(\varphi(x)\) on the tangent cone and clarify its exact relation with the global slope.  For Lipschitz continuous functions we prove that \(\sup_{x\notin S_0} \varphi(x) < 0\) is sufficient for an error bound, and for convex functions on convex sets this condition is also necessary,    In finite dimensions we obtain sharp local results: if \(\varphi(\bar{x}) > 0\) at a solution \(\bar{x}\), then a local error bound holds and the optimal constant is exactly \(1/\varphi(\bar{x})\); if \(\varphi(\bar{x}) = 0\) and a suitable direction exists outside the tangent cone of the solution set, the local error bound fails.   A general estimate relating the directional derivative to the distance from the tangent cone of the solution set is also derived.

\end{minipage}
\end{center}
\section{Introduction}
Error bounds, which quantify the distance from a candidate point to the solution set of a given problem in terms of a residual function, play a fundamental role in variational analysis and optimization. Since the seminal work of Hoffman~\cite{Hoffman} on linear inequality systems, the study of error bounds has attracted lasting interest due to their importance in the convergence analysis of algorithms, sensitivity and stability of solutions, and constraint qualifications; see, e.g.,~\cite{Pang1997, FacchineiPang,  AC2004,KL1999,Kruger,KLT2018,NgZheng,Penot2019,Hu2007,Zheng2012,VanNgai2010,Ngai2015, Wu2002, Li2009} and the references therein. Error bounds have been established for various structured problems, including convex polynomials over polyhedral sets~\cite{Ngai2015}, semi-infinite convex systems~\cite{VanNgai2010}, quasi-subsmooth inequalities~\cite{Zheng2012}, generalized power cones~\cite{Lin2024}, and sign-constrained Stiefel manifolds~\cite{Chen2025}. Moreover, error bounds are fundamental in the analysis of merit (or gap) functions for variational inequalities, complementarity problems, and vector optimization problems, as they guarantee that such functions faithfully reflect the distance to the solution set~\cite{YaoZheng2018, HMTZ2020, Hearn,Fukushima,FacchineiPang, Li2009, DKG2017, LM2010,TanabeNew}.

Let \( X \) be a   normed vector space.   Given an extended-real-valued function \(f\) and its sublevel set \(S_f = \{x : f(x) \le 0\}\), the global error bound property
\begin{equation}\label{EWJ1}
	d (x, S_f) \le \tau [f(x)]_+   \quad \forall x \in  X.
\end{equation}
provides a linear estimate of the distance to the feasible or optimal set,  where  $\tau > 0$, 
$	d (x, S_f ) = \inf \{ \|x - y\| \mid y \in S_f  \}	$	and $	[f(x)]_+ = \max \{ f(x), 0 \}.$
More generally, for a nonnegative function \(\theta\) on a set \(A\) with zero set \(S_0 = \{x \in A : \theta(x) = 0\}\), one seeks \(\tau >0\) such that \(\operatorname{dist}(x,S_0) \le \tau \theta(x)\) for all \(x \in A\). This formulation naturally encompasses classical error bounds for constraint systems \(f(x)\le 0\) by taking \(\theta(x) = [f(x)]_+\) and also accommodates merit functions used in vector optimization~\cite{LiuNgYang2009}.

A powerful approach to characterizing error bounds is through descent conditions and slopes.  
Az\'{e} and Corvellec~\cite{AC2004} characterized global and local error bounds for lower semicontinuous functions on complete metric spaces using the strong slope.   Kruger~\cite{Kruger} developed a comprehensive theory of error bounds and metric subregularity using various primal and subdifferential slopes. Stability of error bounds under perturbations has been studied in~\cite{KLT2018,KNT2010}. These slope-based characterizations are extremely general and have laid a solid theoretical foundation for the understanding of error bounds.

While slope-based criteria are broadly applicable, they do not immediately provide explicit geometric conditions. Directional derivatives, in contrast, offer a more direct and geometric way to measure descent, and they are intimately linked to slopes.  For convex functions on reflexive Banach spaces, Ng and Zheng~\cite{NgZheng} established that a global error bound is equivalent to the existence, for every point outside the solution set, of a direction in which the directional derivative is strictly negative and bounded away from zero.   For linear inequality systems, Ng and Yang~\cite{NgYang2002} derived error bound criteria via angles between subspaces and polyhedral cone properties, and obtained exact characterizations in Hilbert spaces.     Hu~\cite{Hu2007} provided characterizations of local and global error bounds for convex inequalities in Banach spaces using directional derivatives and subdifferentials. 
More recently, Wei, Théra, and Yao~\cite{WeiTheraYao2024} gave primal characterizations of
the stability of error bounds, both local and global, for semi-infinite convex constraint systems in Banach spaces  using directional derivatives, while their subsequent work~\cite{WeiTheraYao2025} investigated the stability of error bounds under perturbations for convex functions on Banach spaces, providing explicit formulae for error bound moduli in terms of directional derivatives. However, these characterizations rely heavily on convexity, and extending them to nonconvex, nonsmooth functions without losing the directional derivative formulation has not been fully explored.

In the present paper we develop new characterizations of both global and local error bounds for general  functions using directional derivatives and tangent cones, without imposing convexity or linear structure. 
Our contributions are threefold.

(1) Global error bound equivalences.  In Section~3 we prove that the global error bound for \(\theta\) on \(A\) is equivalent to a uniform descent condition, to the existence of a descent direction from every non-solution point, to a strengthened descent estimate involving the distance to \(S_0\), and to a uniform positivity of the global slope 
\[
m(x) = \sup_{\substack{y\in A\\ \theta(y)<\theta(x)}} \frac{\theta(x)-\theta(y)}{\|x-y\|}.
\]
These equivalences (Theorem~\ref{TTWCJK}) are established using only the lower semicontinuity of \(\theta\) and the completeness of \(A\) via Ekeland's variational principle, making them valid for nonconvex functions as well.  We further connect the error bound with an extended-domain infimum, perturbation stability, a global estimate involving the distance to \(A\), Hausdorff stability of sublevel sets, and an inverse-sublevel-set estimate (Theorem~\ref{XLWCJK}). These results  enrich the perturbation and stability viewpoints presented in~\cite{KLT2018,Kruger,KNT2010,VanNgai2010} and extend classical criteria to a fully nonconvex setting.

(2) Directional derivative characterizations.  In Section~4 we introduce the minimal directional derivative on the unit sphere of the tangent cone:
\[
\varphi(x) := \inf\bigl\{ \theta'(x;v) : v \in T_A(x),\; \|v\| = 1 \bigr\}.
\]
We clarify the exact relationship between \(\varphi(x)\) and the global slope \(m(x)\). For Lipschitz continuous \(\theta\) we show that \(m(x) \ge -\varphi(x)\), while for convex \(\theta\) on a convex set the reverse inequality holds (Theorem~\ref{ThyvQ}). Building on this, we prove that if \(\theta\) is Lipschitz and \(\sup_{x\in A\setminus S_0} \varphi(x) < 0\), then an error bound holds (Theorem~\ref{TVhCJK}); when \(\theta\) and \(A\) are convex, this condition is also necessary, thereby subsuming the directional derivative criterion of Ng and Zheng~\cite{NgZheng}.  

(3) Local error bounds and optimal constants.  In the finite-dimensional setting we obtain sharp local results. If \(\varphi(\bar x) > 0\) at a point \(\bar x \in S_0\), then \(\theta\) has a local error bound at \(\bar x\) and the optimal local error bound constant is exactly \(1/\varphi(\bar x)\) (Theorem~\ref{TYXWs}). Conversely, if \(\varphi(\bar x) = 0\) and there exists a direction in the tangent cone of \(A\) but not in the tangent cone of \(S_0\) with zero directional derivative, the local error bound fails. We also derive a necessary condition involving the tangent cones (Corollary  \ref{CtgXEQ}) and a general lower estimate of the directional derivative in terms of the distance to the tangent cone of \(S_0\) (Theorem \ref{TnbEQ}). These findings  provide a geometric interpretation of the exact error bound constant.

The remainder of the paper is organized as follows. Section~2 recalls basic concepts and notation, including the tangent cone and Ekeland's variational principle. Section~3 is devoted to the equivalent characterizations of the global error bound (Theorems~3.1 and~3.2) and the relationship between local and global error bounds under compactness.
Section~4 is devoted to the directional derivative characterizations, the relationship with the slope, and the analysis of local error bounds and optimal constants.

\section{Preliminaries}
Throughout this paper, unless otherwise specified,    let \( X \) be a   normed
vector space    and   \( A \subseteq X \) be a nonempty set.    Assume that   \( \theta : X \to \mathbb{R} \cup \{+\infty\} \) is  a proper function such that \( S_0 := \{ x \in A : \theta(x) = 0 \} \neq \emptyset \) and for all \( x \in A \), \( \theta(x) \geq 0 \). We say that \( \theta \) has an error bound on \( A \) if there exists \( \tau > 0 \) such that
\begin{equation}\label{EWJ2}
	d (x, S_0) \leq \tau \theta(x), \quad \forall x \in A.
\end{equation}

\begin{remark}
	The classical error bound \eqref{EWJ1} is a special case of the general formulation \eqref{EWJ2} studied in this paper. 
	Indeed, take \(A = X\) and for any \(x \in X\) define \(\theta(x) = [f(x)]_+ = \max\{f(x), 0\}\). 
	Then clearly \(S_f = \{x \in X \mid f(x) \leq 0\} = \{x \in A : \theta(x) = 0\} = S_0\), and \eqref{EWJ2} reduces to \eqref{EWJ1}.
\end{remark}

$B(x_0,r)$ denotes the open ball with center $x_0 \in X$ and radius $r > 0$. 
For   nonempty sets \(Q, D \subseteq X\),  
the Hausdorff distance between $Q$ and $D$ is defined by
$$d_H \left( {Q,D} \right): = \max \left\{ {g\left( {Q,D} \right),g\left( {D,Q} \right)} \right\},$$
where $g\left( {Q,D} \right): = \mathop {\sup }\limits_{x \in Q} d\left( {x,D} \right)$ and $d(x,D) := \inf_{y \in D} \|x - y\|$.

\begin{definition}
	$\theta$ is said to have a  local error bound at $x_0 \in S_0$, if there exist $\tau > 0$  and $\delta > 0$ such that
	\[
	d(x, S_0) \leq \tau \theta (x), \qquad \forall x \in A \cap B(x_0,\delta).
	\]
	
\end{definition}
In keeping with the distinction from local error bounds, the term error bound is occasionally called a global error bound.

\begin{definition}  
	For a nonempty set $D \subseteq X$ and $x \in D$, the  tangent cone  of  $D$ at  $x$ is defined  by  
	$$T_D(x)  := \bigl\{ d \in X : \exists\, t_k \downarrow 0,\; \exists\, d_k \to d \text{ with } x + t_k d_k \in D \bigr\}.$$
\end{definition}
\begin{remark}  It is straightforward to verify that $T_D(x)$ is closed, without assuming that $D$ is closed.   \end{remark}

\begin{remark} \label{RtDFxl}    \cite{AE1984}    Let  $D \subseteq X$ be a closed convex set  and $x \in D$.  Then  $$  T_D(x) = \cl \left( \bigcup_{t > 0} t(D - x) \right). $$
\end{remark}

\begin{lemma}[Ekeland's variational principle]\label{ELdaz}
	Let $(M, d)$ be a complete metric space and $f: M \to \mathbb{R} \cup \{+\infty\}$ a proper lower semicontinuous function bounded below. If $x_0 \in \operatorname{dom} f$ and $\varepsilon > 0,\ \lambda > 0$ satisfy
	\[
	f(x_0) \le \inf_{x \in M} f(x) + \varepsilon,
	\]
	then there exists $\bar{x} \in M$ such that:
	\begin{enumerate}
		\item $f(\bar{x}) \le f(x_0)$;
		\item $d(x_0, \bar{x}) \le \lambda$;
		\item For all $x \in M \setminus \{\bar{x}\}$,
		\[
		f(x) > f(\bar{x}) - \frac{\varepsilon}{\lambda} \, d(x, \bar{x}) .
		\]
	\end{enumerate}
\end{lemma}

\section{Equivalent  conditions  for   error  bounds of \(\theta\)}

In this section we establish several equivalent characterizations of the global error bound property \eqref{EWJ2} for the function \(\theta\) defined on a nonempty set \(A\).  
We first introduce the sublevel–set mapping \(\Phi(\cdot)\) and the global slope \(m(\cdot)\), and then prove that the error bound is equivalent to a uniform descent condition, the existence of descent directions,  the strengthened descent    condition,  and a strictly positive lower bound for the slope (Theorem~\ref{TTWCJK}).  
Afterwards, in Theorem~\ref{XLWCJK}, we relate the error bound to an extended–domain infimum, perturbation stability, a global estimate involving the distance to \(A\), Hausdorff stability of the sublevel sets, and an inverse–sublevel–set estimate.  
Together, these results substantially extend classical error bound criteria and illustrate the interplay between metric regularity, descent directions, and slope conditions for general   functions.

The  sublevel–set mapping $\Phi:[0,+\infty)\to   2^A$ is defined by
$$\Phi(t)=\{x\in A:\theta(x)\le t\}, \quad  \forall t \in [0,+\infty).$$

For any \(x\in A\) with \(\theta(x)>0\),  the  global slope at \(x\) is defined by
\[
m(x):=\sup_{\substack{y\in A\\ \theta(y)<\theta(x)}}\frac{\theta(x)-\theta(y)}{\|x-y\|}\in[0,+\infty].
\]

\begin{theorem}  \label{TTWCJK}  Let $X$   be a Banach space  and $A \subseteq X$ be a nonempty closed   set. Assume that $\theta$ is lower semicontinuous on $A$. Then   the following five statements are equivalent:
	\begin{enumerate}
		\item[(i)] $\theta$ has an error bound on $A$: there exists $\tau>0$ such that
		\begin{equation}\label{EBAQJ} 
			d(x,S_0)\le\tau \theta(x), \quad  \forall x\in A.
		\end{equation}		
		\item[(ii)]  There exist  \(\alpha > 0\) and \(\beta \in (0,1)\) such that for every \(x \in A\) one can find \(y \in A\) satisfying
		\begin{equation}\label{EYZXJ} 
			\|x - y\| \le \alpha\,\theta(x) \quad\text{and}\quad \theta(y) \le \beta\,\theta(x).
		\end{equation}	
		\item[(iii)] There exists   \(c>0\) such that for each \(x\in A\setminus S_0\), one can find \(y\in A\) satisfying
		\begin{equation}\label{EevL}
			y\neq x \quad\text{and}\quad
			\theta(y)\le \theta(x)-c\|x-y\| .
		\end{equation}
		\item[(iv)]   There exists   \(\eta>0\) such that for every \(x\in A\setminus S_0\), one can find  \(y\in A\) satisfying
		\[
		\theta(y) \le \theta(x) - \eta\,\|x-y\| \qquad\text{and}\qquad \theta(y) \le \theta(x) - \eta\,d(x,S_0).
		\]
		\item[(v)]   There exists   \(\gamma >0\) such that
		\[
		m(x)\ge \gamma,\quad\forall x\in A\setminus S_0.
		\]
	\end{enumerate}
\end{theorem}
\begin{proof}   	\noindent\textbf{(i) $\Rightarrow$ (ii)}.
	Assume  that   there exists $\tau>0$ such that  (\ref{EBAQJ}) holds.   
	Fix an arbitrary \(x \in A\). 
	
	If \(\theta(x) = 0\), then \(x \in S_0\) and taking \(y = x\) fulfills the requirement.
	
	If \(\theta(x) > 0\),  by   (\ref{EBAQJ}),     there exists \(x' \in S_0\) such that
	\[
	\|x - x'\| < d(x,S_0) + \theta(x) \le \tau\theta(x) + \theta(x) = (\tau+1)\theta(x).
	\]
	Set \(y = x'\),   then \(y \in A\) and \(\theta(y) = 0 \le \beta\theta(x)\) for any \(\beta \in (0,1)\). Thus (ii) holds with constant \(\alpha = \tau+1\) and an arbitrary \(\beta \in (0,1)\).

	\noindent\textbf{(ii) $\Rightarrow$ (i)}.
	Assume there exist \(\alpha>0\) and \(\beta \in (0,1)\) satisfying  (\ref{EYZXJ}). Set \(\tau = \frac{\alpha}{1-\beta}\).
	Fix an arbitrary \(x \in A\). 
	
	If \(\theta(x)=0\), then \(x \in S_0\) and the conclusion is trivial. 
	
	Suppose \(\theta(x) > 0\). Starting from \(x_0 = x\), we recursively construct a sequence \(\{x_n\} \subseteq A\). 
	Assuming \(x_n\) is already obtained, if \(\theta(x_n) = 0\) we stop,  otherwise by condition (\ref{EYZXJ}) there exists \(x_{n+1} \in A\) such that
	\[
	\|x_n - x_{n+1}\| \le \alpha \theta(x_n), \qquad \theta(x_{n+1}) \le \beta \theta(x_n).
	\]
	Iterating these inequalities yields, for all steps,
	\[
	\theta(x_n) \le \beta^n \theta(x), \qquad \|x_n - x_{n+1}\| \le \alpha \beta^n \theta(x).
	\]
	Since \(0 < \beta < 1\), the series \(\sum_{n=0}^\infty \beta^n\) converges, hence \(\{x_n\}\) is a Cauchy sequence. As a closed subset of a Banach space, \(A\) is complete, so there exists  \(\bar{x} \in A\) with \(x_n \to \bar{x}\). 
	By the       lower semicontinuity of \(\theta\), we get 
	\[
	\theta(\bar{x}) \le    	\liminf_{n \to \infty} \theta(x_n) \le   	\liminf_{n \to \infty} \beta^n \theta(x) = 0.
	\]
	Since \(\theta\) is nonnegative on \(A\), we obtain \(\theta(\bar{x}) = 0\), i.e., \(\bar{x} \in S_0\).   Then, we have
	\[ d(x,S_0) \le
	\|x - \bar{x}\| \le \sum_{n=0}^{\infty} \|x_n - x_{n+1}\|
	\le \sum_{n=0}^{\infty} \alpha \beta^n \theta(x)
	= \frac{\alpha}{1-\beta}\,\theta(x) = \tau \theta(x).
	\]
	This establishes (i).
	
	\noindent\textbf{(i) $\Rightarrow$ (iv)}.  Assume that there exists \(\tau>0\) such that   (\ref{EBAQJ}) holds.
	Set \(\eta:=\frac{1}{2\tau}>0\). 
	Take any \(x\in A\setminus S_0\). Then \(\theta(x)>0\) and     it follows from  (\ref{EBAQJ})  that     \(d(x,S_0)\le\tau\theta(x)\). By definition of the distance, there exists \(y\in S_0\) such that
	\begin{equation}\label{EghuB1}
		\|x-y\|<d(x,S_0)+\tau\theta(x)\le 2\tau\theta(x)=\frac{1}{\eta}\theta(x).
	\end{equation}
	Since \(y\in S_0\), we have \(\theta(y)=0\). Combining this with \eqref{EghuB1} gives
	\[
	\theta(y)=0\le \theta(x)-\eta \|x-y\|.
	\]
	Noting that \(d(x,S_0)\le \|x-y\|\), we obtain
	\[
	\theta(y)=0\le \theta(x)- \eta  \|x-y\|   \le \theta(x) - \eta\,d(x,S_0).
	\]

	\noindent\textbf{(iv) $\Rightarrow$ (iii)}.  Suppose that there exists   \(c >0\) such that for every \(x\in A\setminus S_0\), one can find  \(y\in A\) satisfying
	\[
	\theta(y) \le \theta(x) - c \,\|x-y\| \qquad\text{and}\qquad \theta(y) \le \theta(x) - c \,d(x,S_0).
	\]
	It suffices to show that $y\neq x$.
	Since $\theta$ is lower semicontinuous on $A$ and $A$ is closed, it is easy to see that $S_0$ is closed.  Due to \(x\in A\setminus S_0\), we have $d(x,S_0) >0$.  Hence, 
	$$\theta(y) \le \theta(x) - c \,d(x,S_0) <  \theta(x), $$
	which implies $y\neq x$.

	\noindent\textbf{(i) $\Rightarrow$ (v)}.
	Assume that there exists \(\tau>0\) such that   (\ref{EBAQJ}) is true.
	Take an arbitrary \(x\in A\setminus S_0\). Then \(\theta(x)>0\)  and	$	d:=d(x,S_0)>0 .$
	For any \(\varepsilon>0\), there exists \(y_\varepsilon\in S_0\) satisfying
	\begin{equation}\label{Eapprox1}
		\|x-y_\varepsilon\|<d+\varepsilon .
	\end{equation}
	Since \(y_\varepsilon\in S_0\), we have \(\theta(y_\varepsilon)=0<\theta(x)\).  so \(y_\varepsilon\) belongs to the set over which the supremum defining \(m(x)\) is taken.   It follows from  (\ref{Eapprox1})   and    \(\theta(x)>0\)  that
	\begin{equation}\label{Eapprox2}
		\frac{\theta(x)-\theta(y_\varepsilon)}{\|x-y_\varepsilon\|}=\frac{\theta(x)}{\|x-y_\varepsilon\|}>\frac{\theta(x)}{d+\varepsilon}.
	\end{equation}
	From (\ref{EBAQJ}) we obtain \(\theta(x)\ge d/\tau\), and so
	\begin{equation}\label{Eapprox3}
		\frac{\theta(x)}{d+\varepsilon}\ge\frac{d/\tau}{d+\varepsilon}=\frac{1}{\tau}\cdot\frac{d}{d+\varepsilon}.
	\end{equation}
	Therefore, for this \(y_\varepsilon\),  we conclude from  (\ref{Eapprox2})  and  (\ref{Eapprox3})  that
	\[
	m(x)\ge\frac{\theta(x)-\theta(y_\varepsilon)}{\|x-y_\varepsilon\|}>\frac{1}{\tau}\cdot\frac{d}{d+\varepsilon}.
	\]
	Letting \(\varepsilon\to0^{+}\), the right-hand side tends to \(\gamma:=\frac{1}{\tau}\),  hence \(m(x)\ge \gamma\). Since \(x\in A\setminus S_0\) was arbitrary, this proves (v).

	\noindent\textbf{(v) $\Rightarrow$ (iii)}.
	Suppose   that there exists   \(\gamma >0\) such that
	\[
	m(x)\ge \gamma,\quad \forall x\in A\setminus S_0.
	\]
	Set \(c :=\dfrac{\gamma}{2}>0\). For any \(x\in A\setminus S_0\), since \(m(x)\ge \gamma > c\), by the definition of supremum, there exists \(y\in A\) satisfying \(\theta(y)<\theta(x)\) and
	\[
	\frac{\theta(x)-\theta(y)}{\|x-y\|} > c,
	\]
	and so \(\theta(y) < \theta(x) - c \|x-y\|\).   It follows from \(\theta(y)<\theta(x)\)  that    \(y\neq x\). Therefore (iii) holds (with constant \(c =\dfrac{\gamma}{2}\)).
	
	\noindent\textbf{(iii) $\Rightarrow$ (i)}.
	Assume that  there exists   \(c>0\) such that     (\ref{EevL})   holds.
	We will prove that for all \(x\in A\),
	\[
	d(x,S_0)\le \frac{1}{c}\,\theta(x).
	\]
	Suppose, to the contrary, that there exists \(x_0\in A\) such that  
	\begin{equation}\label{Econtra1}
		d(x_0,S_0)>\frac{1}{c}\,\theta(x_0).
	\end{equation}
	This yields  that	 \(\theta(x_0)>0\). 
	Consider the restriction \(f:=\theta|_A\colon A\to[0,\infty)\) on the complete metric space \(A\) (a closed subset of a Banach space). It is clear that \(f\) is lower semicontinuous  and \(\inf_{x \in A} f(x)=0\).  
	Set
	\[
	\varepsilon:=f(x_0)=\theta(x_0)>0,\qquad \lambda:=\frac{\varepsilon}{c}=\frac{\theta(x_0)}{c}>0 .
	\]
	Then \(f(x_0)\le \inf_{x \in A} f(x) +\varepsilon\).    Lemma  \ref{ELdaz} (Ekeland's variational principle)   yields that  there exists \(\bar x\in A\) satisfying:
	\begin{itemize}
		\item[(a)] \(f(\bar x)\le f(x_0)\);
		\item[(b)] \(\|\bar x-x_0\|\le \lambda\);
		\item[(c)] for all \(x\in A\setminus\{\bar x\}\), \(f(x) > f(\bar x)-\dfrac{\varepsilon}{\lambda}\,\|x-\bar x\|\).
	\end{itemize}
	If \(\bar x \in  S_0\), then  it follows from  (b) that
	\[
	d(x_0,S_0)\le \|x_0-\bar x\|\le \lambda = \frac{1}{c}\theta(x_0),
	\]
	which contradicts (\ref{Econtra1}). Hence   \(\bar x \notin  S_0\),  and so   \(\theta(\bar x)=f(\bar x)>0\).   (c) can be rewritten as
	\begin{equation}\label{Econtra2}
		\forall x\in A\setminus\{\bar x\},\qquad \theta(x) > \theta(\bar x) - c\,\|x-\bar x\|,
	\end{equation}
	because \(\varepsilon/\lambda = c\).
	Now apply  (\ref{EevL})   to \(\bar x\in A\setminus S_0\),  there exists \(y_0 \in A\) with \(y_0 \neq \bar x\) such that
	\begin{equation}\label{Econtra3}
		\theta(y_0 ) \le \theta(\bar x) - c\,\|\bar x-y_0\|.
	\end{equation}
	It follows from  (\ref{Econtra2})  that		$\theta(y_0) > \theta(\bar x) - c\,\|y_0-\bar x\|,$
	which contradicts (\ref{Econtra3}). 
	Therefore
	\[
	d(x,E_w)\le \frac{1}{c}\,\theta(x), \quad \forall x\in A,
	\]
	which is exactly statement (i) with \(\tau = 1/c\).  	This completes the proof. 
	\end{proof}

\begin{theorem} \label{XLWCJK}  Consider the following statements:
	\begin{enumerate}
		\item[(i)] $\theta$ has an error bound on $A$: there exists $\tau>0$ such that
		\begin{equation}\label{EBAQJ11} 
			d(x,S_0)\le\tau \theta(x), \quad  \forall x\in A.
		\end{equation}		 
		
		\item[(ii)]  There exists   \(\kappa > 0\) such that  
		$$			d(x, S_0) \le \kappa \, \zeta(x), \quad \forall x \in X, 		$$
		where  $\zeta(x) := \inf_{a \in A} \bigl( \theta(a) + \|x - a\| \bigr)$.
		
		\item[(iii)]  There exists \(\kappa > 0\) such that  
		\begin{equation}\label{EBkkQJ} 
			d(x+e, S_0) \le \kappa (\theta(x) + \|e\|),  \quad  \forall x \in A, \;    \forall e \in X.
		\end{equation}	
		
		\item[(iv)]  There exists  $\gamma>0$ such that  
		\begin{equation}\label{EBkXJ} 
			d(x,S_0)\le \gamma\bigl(\max\{\theta(x),0\}+d(x, A)\bigr),\quad \forall x\in X .
		\end{equation}	
		
		\item[(v)] There exist \(t_0 > 0\) and \(\beta > 0\) such that 
		\[
		d_H(\Phi(t), S_0) \le \beta t,  \quad  \forall t \in [0, t_0].
		\]
		
		\item[(vi)] There exist $\kappa>0$ and an open set $U\supseteq S_0$ such that \begin{equation}\label{EBkmJ} 
			d(x,\Phi(0))\le\kappa\,d(0,\Phi^{-1}(x)),       \quad  \forall       x\in U\cap A. 	
		\end{equation}	 
	\end{enumerate}
	Then the following hold:
	\begin{enumerate}
		\item      (i) $\Leftrightarrow$ (ii), (iii) $\Rightarrow$ (i), (iv) $\Rightarrow$ (i),  (i) $\Rightarrow$ (v) and (i) $\Rightarrow$ (vi). 
		\item  If $\theta$ is Lipschitz continuous on $X$ with constant $L>0$,
		then (i) $\Rightarrow$ (iii) and (i) $\Rightarrow$ (iv).
		\item  If $A$ is bounded, then (v) $\Rightarrow$ (i).
		\item  	If $A$ is compact  and $\theta$ is lower semicontinuous on $A$, then (vi) $\Rightarrow$ (i).  
	\end{enumerate}
\end{theorem}

\begin{proof}
	\noindent\textbf{(i) $\Rightarrow$ (ii)}.
	Suppose there exists \(\tau > 0\) such that   (\ref{EBAQJ11}) is true.
	Take an arbitrary \(x \in X\). For any \(a \in A\),   it follows from  (\ref{EBAQJ11})  that
	\[
	d(x, S_0) \le \|x - a\| + d(a, S_0) \le \|x - a\| + \tau \, \theta(a).
	\]
	Set \(\kappa := \max\{\tau, 1\}\). Then  
	$  \|x - a\| + \tau \, \theta(a) \le \kappa \bigl( \|x - a\| + \theta(a) \bigr),$  and so
	$$d(x, S_0) \le \kappa (\|x - a\| + \theta(a)),  \quad  \forall a \in A.  $$
	Taking the infimum over \(a \in A\) yields  
	\[
	d(x, S_0) \le \kappa \inf_{a \in A} \bigl( \theta(a) + \|x - a\| \bigr) = \kappa \, \zeta(x).
	\]
	Since \(x \in X\) is arbitrary, this proves (ii).

	\noindent\textbf{(ii) $\Rightarrow$ (i)}.
	Suppose there exists \(\kappa > 0\) such that 
	\begin{equation}\label{ErtB1}
		d(x, S_0) \le \kappa \, \zeta(x), \quad  \forall x \in X.
	\end{equation}
	Take an arbitrary \(x \in A\). From the definition of \(\zeta\) we have  
	\[
	\zeta(x) = \inf_{a \in A} \bigl( \theta(a) + \|x - a\| \bigr) \le \theta(x) + \|x - x\| = \theta(x).
	\]
	Substituting this into    (\ref{ErtB1})   gives  
	\[
	d(x, S_0) \le \kappa \, \theta(x), \quad \forall x \in A.
	\]
	
	\noindent\textbf{(iii) $\Rightarrow$ (i)}.
	Suppose  that there exists \(\kappa > 0\) such that  (\ref{EBkkQJ}) holds.   
	Take any \(x \in A\) and set \(e = 0\). From (\ref{EBkkQJ}) we immediately obtain
	\[
	d(x, S_0) = d(x+0, S_0) \le \kappa (\theta(x) + 0) = \kappa \theta(x).
	\]

	\noindent\textbf{(iv) $\Rightarrow$ (i)}.
	Assume that  there exists $\gamma>0$ such that (\ref{EBkXJ}) is satisfied.
	Take any $x\in A$. Then $d(x, A)=0$  and   $\theta(x)\ge0$, and so $\max\{\theta(x),0\}=\theta(x)$. Substituting into (\ref{EBkXJ}) immediately yields
	\[
	d(x,S_0)\le \gamma\,\theta(x),\quad \forall x\in A.
	\]

	\noindent\textbf{(i) $\Rightarrow$ (v)}.
	Suppose that  there exists \(\tau > 0\) such that   (\ref{EBAQJ11})  is satisfied.
	
	Take arbitrary \(t \ge 0\) and \(x \in \Phi(t)\). Since \(x \in A\) and \(\theta(x) \le t\),   (\ref{EBAQJ11})  yields
	\[
	d(x, S_0) \le \tau \theta(x) \le \tau t.
	\]
	Because \(S_0 \subseteq \Phi(t)\), clearly \(\sup_{y \in S_0} d(y, \Phi(t)) = 0\). Hence the Hausdorff distance satisfies
	\[
	d_H(\Phi(t), S_0) = \sup_{x \in \Phi(t)} d(x, S_0) \le \tau t, \quad \forall t \ge 0.
	\]
	Taking \(t_0 = 1\) and \(\beta = \tau\), we have for all \(t \in [0, t_0]\) that \(d_H(\Phi(t), S_0) \le \beta t\).

	\noindent\textbf{(i) $\Rightarrow$ (vi)}.
	By the definition of $\Phi$, clearly $\Phi(0)=\theta^{-1}(0)\cap A= S_0$.
	For a fixed $x\in A$,
	\[
	\Phi^{-1}(x)=\{t\ge0:\theta(x)\le t\}=[\theta(x),+\infty),
	\]
	and	so $d(0,\Phi^{-1}(x))=\inf\{|t-0|:t\ge\theta(x)\}=\theta(x)$.
	Therefore,      (\ref{EBkmJ})    is equivalent to
	\begin{equation}\label{EViiiD}
		d(x, S_0)\le\kappa\,\theta(x),\quad\forall x\in U\cap A.
	\end{equation}
	If (i) holds, taking $U=X$ shows that (vi) holds.

	\noindent\textbf{(i) $\Rightarrow$ (iii)}.
	Suppose there exists \(\tau > 0\)  such that   (\ref{EBAQJ11}) holds. Take arbitrary \(x \in A\) and \(e \in X\).
	
	It follows from  \(x \in A\)  that \(d(x+e, A) \le \|x+e - x\| = \|e\|\).
	For any \(\varepsilon > 0\), choose \(a_\varepsilon \in A\) such that
	\begin{equation}\label{eTHVDD1}
		\|x+e - a_\varepsilon\| \le d(x+e, A) + \varepsilon \le \|e\| + \varepsilon,
	\end{equation}
	Since  \(\theta\) is Lipschitz continuous on \(X\) with constant \(L\),  by (\ref{eTHVDD1}), we have 
	\[
	|\theta(x+e) - \theta(a_\varepsilon)| \le L \|(x+e) - a_\varepsilon\| \le L(\|e\| + \varepsilon),
	\]
	hence
	\begin{equation}\label{eTHVDD2}
		\theta(a_\varepsilon) \le \theta(x+e) + L(\|e\| + \varepsilon).
	\end{equation}
	We conclude from  \(a_\varepsilon \in A\),  (\ref{EBAQJ11})  and  (\ref{eTHVDD2})  that
	\begin{equation}\label{eTHVDD3}
		d(a_\varepsilon, S_0) \le \tau \theta(a_\varepsilon) \le \tau\bigl(\theta(x+e) + L(\|e\| + \varepsilon)\bigr).
	\end{equation}
	It follows from  (\ref{eTHVDD1})  and  (\ref{eTHVDD3})    that
	\[
	\begin{aligned}
		d(x+e, S_0) &\le \|x+e - a_\varepsilon\| + d(a_\varepsilon, S_0) \\
		&\le (\|e\| + \varepsilon) + \tau\bigl(\theta(x+e) + L(\|e\| + \varepsilon)\bigr) \\
		&= \tau \theta(x+e) + (1 + \tau L)(\|e\| + \varepsilon).
	\end{aligned}
	\]
	Letting \(\varepsilon \to 0^+\), we obtain
	\begin{equation}\label{eq:d-xe}
		d(x+e, S_0) \le \tau \theta(x+e) + (1 + \tau L)\|e\|.
	\end{equation}
	Since  \(\theta\) is Lipschitz continuous, \(\theta(x+e) \le \theta(x) + L\|e\|\). Substituting into \eqref{eq:d-xe} yields
	\[
	\begin{aligned}
		d(x+e, S_0) &\le \tau\bigl(\theta(x) + L\|e\|\bigr) + (1 + \tau L)\|e\| \\
		&= \tau \theta(x) + (2\tau L + 1)\|e\|.
	\end{aligned}
	\]
	Due to  \(\theta(x) \ge 0\) and \(\|e\| \ge 0\), we have
	\[
	\tau \theta(x) + (2\tau L + 1)\|e\| \le (\tau + 2\tau L + 1)(\theta(x) + \|e\|).
	\]
	Setting \(\kappa = \tau + 2\tau L + 1\), we obtain
	$	d(x+e, S_0) \le \kappa (\theta(x) + \|e\|),$
	which means that  (\ref{EBkkQJ}) holds.

	\noindent\textbf{(i) $\Rightarrow$ (iv)}.
	Assume  that   there exists $\tau>0$ such that  (\ref{EBAQJ11}) holds.
	Set  
	$ \gamma := \max\{\,\tau,\;1+\tau L\,\}.$
	Let  $x\in X$      be arbitrary.   For   any    $\varepsilon>0$,  there exists $a_\varepsilon\in A$ such that
	\begin{equation}\label{EvgyB1}	\|x-a_\varepsilon\| < d(x, A)+\varepsilon  .        
	\end{equation}
	It follows from    (\ref{EBAQJ11})   and  (\ref{EvgyB1})    that   
	\begin{equation}\label{EvgyB2}
		d(x,S_0) \le \|x-a_\varepsilon\| + d(a_\varepsilon,S_0) < d(x, A)+\varepsilon + \tau\,\theta(a_\varepsilon) .
	\end{equation}
	Since  \(\theta\) is Lipschitz continuous on \(X\) with constant \(L\),  due to (\ref{EvgyB1}),  we have
	\begin{equation} \label{EvgyB3}
		\theta(a_\varepsilon) \le \theta(x) + L\|a_\varepsilon-x\| 
		< \theta(x) + L\bigl(d(x, A) + \varepsilon\bigr).  
	\end{equation}
	
	Now we distinguish two cases according to the sign of $\theta(x)$.
	
	\noindent\textbf{Case 1: $\theta(x)\ge 0$.}
	Then $\max\{\theta(x),0\}=\theta(x)$. Substituting (\ref{EvgyB3}) into (\ref{EvgyB2}),
	\[
	\begin{aligned}
		d(x,S_0) &< d(x, A)+\varepsilon + \tau\bigl(\theta(x)+L(d(x, A)+\varepsilon)\bigr) \\
		&= \tau\theta(x) + (1+\tau L)d(x, A) + (1+\tau L)\varepsilon.
	\end{aligned}
	\]
	Since $\tau\le\gamma$, $1+\tau L\le\gamma$,     $\theta(x)\ge0$  and  $  d(x, A)\ge0$, we have
	\[
	\tau\theta(x) + (1+\tau L) d(x, A) \le \gamma\theta(x) + \gamma d(x, A) = \gamma\bigl(\theta(x)+d(x, A)\bigr).
	\]
	Thus
	$	d(x,S_0) < \gamma\bigl(\theta(x)+d(x, A) \bigr) + (1+\tau L)\varepsilon.$
	
	\noindent\textbf{Case 2: $\theta(x)<0$.}
	Then $\max\{\theta(x),0\}=0$. Due to (\ref{EvgyB3})  and   $\theta(x)<0$, we have
	$\theta(a_\varepsilon) < L\bigl(d(x, A)+\varepsilon\bigr)$.
	Substituting into (\ref{EvgyB2}), we have
	\[
	\begin{aligned}
		d(x,S_0) &< d(x, A) +\varepsilon + \tau \cdot L\bigl(d(x, A) +\varepsilon\bigr) \\
		&= (1+\tau L)d(x, A) + (1+\tau L)\varepsilon.
	\end{aligned}
	\]
	Since $1+\tau L\le\gamma$ and in this case $\max\{\theta(x),0\}+d(x, A) = d(x, A)$, we obtain
	\[
	d(x,S_0) < \gamma\bigl(\max\{\theta(x),0\}+d(x, A) \bigr) + (1+\tau L)\varepsilon.
	\]
	
	Combining both cases, for any $\varepsilon>0$ we have
	\[
	d(x,S_0) < \gamma\bigl(\max\{\theta(x),0\}+d(x, A) \bigr) + (1+\tau L)\varepsilon.
	\]
	Letting $\varepsilon\to0^+$, we get
	$	d(x,S_0) \le \gamma\bigl(\max\{\theta(x),0\}+d(x, A) \bigr).	$
	This means that (\ref{EBkXJ}) is true.

	\noindent\textbf{(v) $\Rightarrow$ (i)}.
	Suppose there exist \(t_0 > 0\) and \(\beta > 0\) such that 
	\begin{equation}\label{EthhB1}
		d_H(\Phi(t), S_0) \le \beta t,  \quad  \forall t \in [0, t_0].
	\end{equation}

	Take any \(x \in A\) and set \(t = \theta(x)\). We distinguish two cases:
	
	If \(t \le t_0\), then \(x \in \Phi(t)\). By   (\ref{EthhB1}),  we have 
	\[
	d(x, S_0) \le \sup_{z \in \Phi(t)} d(z, S_0) = d_H(\Phi(t), S_0) \le \beta t = \beta \theta(x).
	\]
	
	If \(t > t_0\), since \(A\) is bounded, let \(\eta = \operatorname{diam}(A) = \sup_{a,b \in A} \|a - b\| < +\infty\). Because \(x \in A\) and \(S_0 \subseteq A\), we have \(d(x, S_0) \le \eta\). Combining this with \(t > t_0\) yields
	\[
	d(x, S_0) \le \eta < \frac{\eta}{t_0} t = \frac{\eta}{t_0} \theta(x).
	\]
	
	Setting \(\tau = \max\{\beta, \eta/t_0\}\), we obtain \(d(x, S_0) \le \tau \theta(x)\) for all \(x \in A\).

	\noindent\textbf{(vi) $\Rightarrow$ (i)}.
	Assume $A$ is compact   and there exist $\kappa>0$ and an open set $U\supseteq S_0$ such that
	(\ref{EBkmJ})  is satisfied.   
	Set $H:=A\setminus U$.   Since   $A$ is compact, $U$ is open and $S_0 \subseteq U$,  we get that  $H$ is compact   and $H\cap S_0= \emptyset$.
	$\theta$ is lower semicontinuous   on the compact set $H$, hence attains its minimum $m:=\min_{x\in H}\theta(x)$.
	If $m=0$, then there exists $\bar x\in H$ with $\theta(\bar x)=0$. This yields   that   $\bar x\in S_0$, contradicting $H\cap S_0= \emptyset$. Therefore $m>0$.
	
	For any $x\in H$, since $A$ is bounded, $\beta =\operatorname{diam}(A)<\infty$, and so $d(x,S_0)\le \beta$. Together with $\theta(x)\ge m$, we have 
	\begin{equation} \label{EHYZ2}
		d(x,S_0)\le \frac{\beta}{m}\,\theta(x).  
	\end{equation}
	
	Take any $x\in A$.  If $x\in U$, noting that  (\ref{EBkmJ})  is  equivalent to (\ref{EViiiD}), we get  
	$d(x,S_0)\le\kappa\,\theta(x)$.
	
	If $x\in H$,  it follows from  (\ref{EHYZ2}) that $d(x,S_0)\le (\beta/m)\,\theta(x)$.
	Taking $\tau:=\max\{\kappa,\; \beta/m\}$, the global error bound holds:
	\[
	d(x,S_0)\le\tau\,\theta(x),\qquad \forall x\in A.
	\]
	This completes the proof. 
\end{proof}

\begin{remark}   		 
	Theorem~\ref{TTWCJK} gives five equivalent characterizations of error bounds for a lower semicontinuous nonnegative function $\theta$ on a closed set $A$ in a Banach space: the direct distance estimate (i), the uniform descent condition (ii), the descent direction condition (iii),    the strengthened descent    condition (iv)  and the global slope lower bound condition (v). The core tool behind this result is the Ekeland variational principle.
	Theorem~\ref{XLWCJK} gives six equivalent conditions for error bounds, including: an estimate for an extension function $\zeta$ (ii), a directional perturbation estimate (iii), a global estimate with a distance term (iv), Hausdorff stability of level sets (v), and an inverse mapping estimate (vi). These conditions connect error bounds with the concepts of  stability  and  regularity.  Below we compare Theorems \ref{TTWCJK} and \ref{XLWCJK} with the results in   \cite{LiuNgYang2009, NgZheng, KL1999, NgYang2002, AC2004, Kruger, KLT2018}.

	Theorems \ref{TTWCJK} and \ref{XLWCJK}  do not require convexity, linearity, or any cone ordering. In contrast, Liu, Ng and Yang \cite{LiuNgYang2009} study merit functions tailored to vector optimization problems with a convex feasible set, a linear map, and a closed convex cone. Their error‐bound results rely on linear regularity of a pair of sets, which is specific to the convex/cone structure.
	Our slope‐based conditions (Theorem \ref{TTWCJK}) and the various stability criteria (Theorem \ref{XLWCJK}) are more general and reduce to their setting when \(\theta\) is chosen as their merit function.  Thus, our results offer a unified framework that complements their convex‐geometric approach and applies to nonconvex or non‐cone‐ordered problems as well.

	Ng and Zheng \cite{NgZheng} studied error bounds for lower semicontinuous functions using directional derivatives. Their Theorem 2.5 gives a sufficient condition \(\underline{d}^+f(x)(h_x)\le -\delta\) for general functions. However, their precise characterizations—Theorems 3.1 and 3.3—are established for convex functions on reflexive Banach spaces. In particular, Theorem 3.1 of \cite{NgZheng} shows that for convex functions, a global error bound is equivalent to the existence, for each \(x\notin S\), of a unit direction \(h_x\) with \(d^+f(x)(h_x)\le -1/\tau\) (condition (v) of Theorem 3.1 in \cite{NgZheng}). Our condition (ii) in Theorem~\ref{TTWCJK} adopts a different perspective: instead of requiring a directional derivative bound, it imposes a uniform descent condition—namely, from every \(x\in A\) one can move to some \(y\in A\) with \(\|x-y\|\le \alpha\theta(x)\) and \(\theta(y)\le \beta\theta(x)\).    Both conditions share the same underlying idea, as both require the existence of a feasible descent direction from every non-solution point.         Importantly, the proof of  (i) $\Leftrightarrow$ (ii) (Theorem~\ref{TTWCJK}) does not rely on the convexity   of \(\theta\),  it uses only the nonnegativity of \(\theta\) on \(A\) and standard metric arguments. Thus condition (ii) is applicable to general functions, in contrast to the directional derivative criterion in \cite{NgZheng}, which is specific to convex functions.  Moreover,  condition (ii) provides an equivalent characterization of the global error bound, without requiring differentiability, convexity, or explicit computation of directional derivatives.

	Klatte and Li  	\cite{KL1999}  studied asymptotic constraint qualifications for convex inequalities and showed equivalence among three conditions: bounded excess, Slater condition together with the asymptotic constraint qualification, and positivity of normal directional derivatives. Their work concerns general convex constraint systems, not general functions. Condition (v) (sublevel set Hausdorff stability) in our theorem (Theorem~\ref{XLWCJK}) is conceptually related to their bounded excess condition, while our characterization via directional derivatives in Theorem  \ref{TVhCJK} ---\(\sup_{x\in A\setminus S_0}\varphi(x)<0\)---parallels their positivity of normal directional derivatives.
	
	Ng and Yang   	\cite{NgYang2002}   studied error bounds for abstract linear inequality systems \(Ax\le b\) in Banach spaces ordered by closed convex cones. They proved that if \(C\) is polyhedral, the system always has an error bound, and characterized error bounds via angles between subspaces in Hilbert spaces. Their results are restricted to linear inequality systems, whereas our theorem applies to general functions, providing a much broader framework.

	Az\'e and Corvellec   \cite{AC2004}   characterized global and local error bounds for lower semicontinuous functions on complete metric spaces using the strong slope \(|\nabla f|\). In their Proposition~3.1, they proved that when \(f\) is convex, the strong slope admits the explicit representation
	\[
	|\nabla f|(x) = \sup_{f(z) < f(x)} \frac{f(x) - f(z)}{\|x - z\|},
	\]
	which is exactly the definition of the slope \(m(x)\) appearing in our condition (v) (Theorem~\ref{TTWCJK}). Condition (v) in our theorem---\(m(x)\ge \gamma\) with \(m(x)=\sup_{\theta(y)<\theta(x)}(\theta(x)-\theta(y))/\|x-y\|\)---is therefore a direct analogue of their strong slope condition, adapted to the   lower semicontinuous function \(\theta\) on the closed set \(A\). It is important to note, however, that while the slope formula of Az\'e and Corvellec relies on convexity, our proof of the equivalences (i) \(\Leftrightarrow\) (iii) \(\Leftrightarrow\) (v) (Theorem~\ref{TTWCJK}) uses neither convexity nor concavity of the function. In fact, these implications follow from purely metric and variational arguments (Ekeland's principle) and hold for any lower semicontinuous function on a complete metric space. Thus, our result   provides a general principle that subsumes both convex and concave cases.

	In his foundational paper, Kruger   \cite{Kruger}   develops a systematic theory of error bounds and metric subregularity for set-valued mappings between general metric or Banach spaces. The criteria are expressed in terms of various primal and subdifferential slopes, and they are primarily \emph{local} in nature. In contrast, Theorems \ref{TTWCJK} and \ref{XLWCJK} provide a comprehensive \emph{global} characterization for the general function  $\theta$. 
	Notably, conditions (iii) and (v) in Theorem \ref{TTWCJK}---the descent and slope conditions---are direct analogues of the primal slope criteria in    \cite{Kruger}, which are defined on arbitrary metric spaces. Since the proofs of these two conditions do not depend on concavity and convexity, they effectively transplant the general global error bound slope criteria to the   set \(A\) with \(S_0\) as the zero-level set.   Moreover, condition (iii) (perturbation stability) and condition (vi) (inverse sublevel-set characterization) in Theorem \ref{XLWCJK} are new global stability features, enriching the framework of    \cite{Kruger} which mainly focused on local slopes.
	
	The perturbation paper \cite{KLT2018} extends the developments in Kruger, Ngai and Th\'era      \cite{KNT2010} and characterizes the stability of local and global error bounds under data perturbations in the Banach space setting. It introduces new concepts of arbitrary, convex and linear perturbations of the function defining the constraint system, and interprets the characterizations as estimates of the ``radius of error bounds''. 	Theorem~\ref{XLWCJK} incorporates perturbation-type characterizations in conditions (iii) and (vi): (iii) asserts stability of the distance to \(S_0\) under arbitrary displacements \(e\in X\), while (vi) expresses the error bound in terms of the inverse image of the sublevel mapping \(\Phi\). These conditions are novel  and directly mirror the perturbation philosophy of  \cite{KLT2018}. However, whereas the perturbation paper focuses on general convex functions and their subdifferentials, Theorem~\ref{XLWCJK} establishes these perturbation characterizations for the general function  $\theta$, leveraging its Lipschitz continuity and compactness. Furthermore, condition (v)---the Hausdorff distance between sublevel sets and \(S_0\)---provides a quantitative stability estimate that is analogous to the ``radius of error bounds'' interpretation in the perturbation paper \cite{KLT2018}.
\end{remark}

We now turn to the relationship between local and global error bounds.

\begin{theorem} 
	\label{thm:localglobal}
	Let $A\subset X$ be a nonempty  compact  set.  Assume that $\theta$ is lower semicontinuous on $A$. 
	Then the following three statements are equivalent:
	\begin{enumerate}
		\item[(i)]  $\theta$ has a global error bound  on $A$: there exists  $\tau>0$ such that
		\[
		d(x,S_0)\le \tau\,\theta(x),\quad \forall x\in A .
		\]
		\item[(ii)]  $\theta$ has a uniform local error bound: there exist  $\tau>0$ and $\delta>0$ such that for  every $\bar x\in S_0$,
		$$	d(x, S_0)\le \tau\,\theta(x),\quad \forall x\in A\cap B(\bar x,\delta).	$$
		\item[(iii)] $\theta$ has a local error bound at each $\bar x\in S_0$.
	\end{enumerate}
\end{theorem}

\begin{proof}
	The implications (i) $\Rightarrow$ (ii) and (ii) $\Rightarrow$ (iii) follow directly from the definitions.
	
	(iii) $\Rightarrow$ (i).
	For any $\bar x\in S_0$, since $\theta$ has a local error bound at $\bar x$, there exist  $\delta_{\bar x} >0$ and $\kappa_{\bar x}>0$ such that
	\[
	d(x,S_0)\le \kappa_{\bar x}\,\theta(x), \quad \forall x\in A\cap B(\bar x,\delta_{\bar x}).
	\]
	Since    $\theta$ is lower semicontinuous on $A$, we obtain that   $S_0$ is closed.   Because $S_0 \subseteq A$ and $A$ is compact,  we get that  $S_0$ is compact. There exist finitely many points $\bar x_1,\dots,\bar x_m\in S_0$ such that
	\[
	S_0 \subseteq\bigcup_{i=1}^{m}B(\bar x_i,\delta_i/2),\qquad \delta_i:=\delta_{\bar x_i}.
	\]
	Set $\delta:=\min_i\delta_i/2>0$ and $\kappa:=\max_i\kappa_{\bar x_i}$. Now take any $x\in A$ satisfying $d(x,  S_0)<\delta$. By the definition of distance, there exists $y\in S_0$ with $\|x-y\|<\delta$. Since $y$ belongs to some $B(\bar x_i,\delta_i/2)$, we have
	\[
	\|x-\bar x_i\|\le\|x-y\|+\|y-\bar x_i\|<\delta+\frac{\delta_i}{2}\le\delta_i,
	\]
	hence $x\in B(\bar x_i,\delta_i)$, which yields
	$	d(x,S_0)\le \kappa_{\bar x_i}\,\theta(x)\le\kappa\,\theta(x).$
	Thus,
	\begin{equation} \label{EglZS1}
		d(x,S_0)\le \kappa\,\theta(x),\qquad \forall x\in A,\ d(x,S_0)<\delta.  
	\end{equation}

	Now take any $x\in A$. We distinguish two cases.
	
	\noindent\textbf{Case 1:} $d(x,S_0)<\delta$.  From (\ref{EglZS1}) we immediately obtain
	$	d(x,S_0)\le \kappa\,\theta(x).$
	
	\noindent\textbf{Case 2:} $d(x,S_0)\ge \delta$. Set
	$	H:=\bigl\{x\in A : d(x,S_0)\ge \delta\bigr\}.$
	Since	$A$ is compact and $d(\cdot,S_0)$ is continuous, we get that $H$ is compact. Clearly $H\cap S_0 =\emptyset$ and $\theta(x)>0$ for all $x\in H$. 
	This together with the lower semicontinuity of $\theta$  implies that
	\[
	m:=\min_{x\in H}\theta(x)>0.
	\]
	On the other hand, $A$ is bounded, and its diameter $\beta:=\operatorname{diam}(A)=\sup_{u,v\in A}\|u-v\|<+\infty$. For $x\in H$, we have $d(x,S_0)\le \beta$, which together with $\theta(x)\ge m$ yields
	\[
	d(x,S_0)\le \beta \le \frac{\beta}{m}\,\theta(x).
	\]
		Combining the two cases, set
	$	\tau := \max\Bigl\{\kappa,\;\frac{\beta}{m}\Bigr\}, $
	then for every $x\in A$,
	$	d(x,S_0)\le \tau \,\theta(x).$   	This completes the proof. 
\end{proof}

 \section{Characterizing error bounds via directional derivatives and tangent cones}

In this section we take a  geometric approach: we introduce the directional derivative of \(\theta\) and the tangent cone to the feasible set \(A\), and we study how these objects can be used to formulate transparent and computationally tractable criteria for both global and local error bounds. 
The central quantity in our analysis is the minimal unit‑sphere directional derivative \(\varphi(x)\) on the tangent cone, which will be shown to be closely related to the global slope \(m(x)\) and to provide sharp estimates for the optimal error bound constant.

Recall that the directional derivative of \(\theta\) at \( x \in X \) in the direction \( v \in X \), denoted \(\theta'(x;v)\), is defined as follows:
\[
\theta'(x;v) = \lim_{t \downarrow 0} \frac{\theta (x + tv) - \theta (x)}{t}.
\]

\begin{theorem}  \label{JBXZDS} 
	Let $x \in X$.	Assume that  $\theta$ is concave and   Lipschitz continuous on $X$ with constant $L$.
	Then	the following statements hold:
	\begin{itemize}
		\item[(i)]  For   any direction $v \in X$, the directional derivative  $	\theta'(x; v)$  exists. 
		\item[(ii)] The map $v \mapsto \theta'( x; v)$ is a Lipschitz continuous   function on $X$ with the   constant $L$.
		\item[(iii)] For all $v_1, v_2 \in X$ and $\alpha \ge 0$,
		\[
		\theta'(x; v_1+v_2) \ge \theta'( x; v_1) + \theta'( x; v_2), \quad
		\theta'( x; \alpha v_1) = \alpha \,\theta'( x; v_1).
		\]
		In		particular,   the map $v \mapsto \theta'( x; v)$ is a   concave function on $X$.
	\end{itemize}
\end{theorem}

\begin{proof}
	Since $\theta$ is concave and   Lipschitz continuous on $X$ with constant $L$,  we have 
	\begin{equation}\label{eq:lip}
		|\theta(x)-\theta(y)|\le L\|x-y\|,\quad \forall x,y\in X,
	\end{equation}
	and for any $x,y\in X$ and $\lambda\in[0,1]$,
	\begin{equation}\label{eq:concave}
		\theta\bigl((1-\lambda)x+\lambda y\bigr)\ge (1-\lambda)\theta(x)+\lambda\theta(y).
	\end{equation}
	
	(i).	Fix an arbitrary point $x\in X$ and a direction $v\in X$.
	For $t>0$,  define the difference quotient
	\[
	Q(t):=\frac{\theta(x+tv)-\theta(x)}{t}.
	\]
	
	We first show that $ Q(\cdot)$ is nonincreasing on $(0,\infty)$. For $0<t_1<t_2$, set $\lambda:=\frac{t_1}{t_2}\in(0,1)$, then
	\[
	x+t_1 v = (1-\lambda)x+\lambda(x+t_2 v).
	\]
	Using the concavity \eqref{eq:concave} of $\theta$, we have 
	$	\theta(x+t_1 v) \ge (1-\lambda)\theta(x)+\lambda\theta(x+t_2 v), $
	and so
	\[
	\theta(x+t_1 v)-\theta(x) \ge \lambda\bigl(\theta(x+t_2 v)-\theta(x)\bigr)
	= \frac{t_1}{t_2}\bigl(\theta(x+t_2 v)-\theta(x)\bigr).
	\]
	Dividing both sides by $t_1$ gives
	\[
	Q(t_1) = \frac{\theta(x+t_1 v)-\theta(x)}{t_1}
	\ge \frac{\theta(x+t_2 v)-\theta(x)}{t_2} = Q(t_2).
	\]
	
	On the other hand, by the Lipschitz condition \eqref{eq:lip}, $Q(t)\le  L\|v\|$ for all $t>0$,
	so the directional derivative 
	\[
	\theta'(x;v)  =\sup_{t  > 0}Q(t) =\lim_{t\downarrow0}Q(t) =\lim_{t\downarrow0}\frac{\theta(x+t v)-\theta(x)}{t}
	\]
	exists and is finite.

	(ii).	For any $v_1,v_2\in X$ and $t>0$, by \eqref{eq:lip},
	\[
	\begin{aligned}
		\bigl|Q_{v_1}(t)-Q_{v_2}(t)\bigr|
		&= \frac{\bigl|\theta(x+t v_1)-\theta(x+t v_2)\bigr|}{t}\\
		&\le \frac{L\|t v_1-t v_2\|}{t} = L\|v_1- v_2\|.
	\end{aligned}
	\]
	Letting $t\downarrow0$, we obtain
	\[
	\bigl|\theta'(x;v_1)-\theta'(x;v_2)\bigr| \le L\|v_1- v_2\|,\qquad \forall v_1, v_2\in X.
	\]

	(iii).	Let $\alpha\ge0$. If $\alpha=0$, then $\theta'(x;0)=\lim_{t\downarrow0}0=0$. If $\alpha>0$, make the change of variable $s=\alpha t$; as $t\downarrow0$, $s\downarrow0$, and
	
	\begin{eqnarray*}   	\theta'(x;\alpha v) &=&  \lim_{t\downarrow0}\frac{\theta(x+t\alpha v)-\theta(x)}{t}
		= \alpha \lim_{t\downarrow0}\frac{\theta(x+(\alpha t)v)-\theta(x)}{\alpha t}   \\
		&=& \alpha \lim_{s\downarrow0}\frac{\theta(x+s v)-\theta(x)}{s}
		= \alpha\,\theta'(x;v).
	\end{eqnarray*}

	For any $v_1,v_2\in X$,   it follows from  (\ref{eq:concave})  that
	\[
	\begin{aligned}
		\theta\bigl(x+t(v_1+v_2)\bigr)
		&= \theta\Bigl(\frac{1}{2}(x+2tv_1)+\frac{1}{2}(x+2tv_2)\Bigr)\\
		&\ge \frac{1}{2}\theta(x+2tv_1) + \frac{1}{2}\theta(x+2tv_2).
	\end{aligned}
	\]
	Subtract $\theta(x)$ and divide by $t>0$:
	\[
	\frac{\theta(x+t(v_1+v_2))-\theta(x)}{t}
	\ge \frac{\theta(x+2tv_1)-\theta(x)}{2t} + \frac{\theta(x+2tv_2)-\theta(x)}{2t}.
	\]
	Letting $t\downarrow0$, the two terms on the right tend to $\theta'(x;v_1)$ and $\theta'(x;v_2)$, respectively,
	while the left-hand side tends to $\theta'(x;v_1+v_2)$. This yields the superadditivity
	$$		\theta'(x;v_1+v_2) \ge \theta'(x;v_1) + \theta'(x;v_2),\qquad \forall v_1,v_2\in X.	$$
	Together with positive homogeneity, for any $\lambda\in[0,1]$,
	\[
	\theta'(x;\lambda v_1+(1-\lambda) v_2)
	\ge \theta'(x;\lambda v_1) + \theta'(x;(1-\lambda) v_2)
	= \lambda\theta'(x;v_1) + (1-\lambda)\theta'(x; v_2).
	\]
	Hence the map $d\mapsto\theta'(x;v)$ is concave.  	This completes the proof. 
\end{proof}

Analogous to Theorem  \ref{JBXZDS}, we obtain the following theorem.

\begin{theorem}  \label{JBXCOS} 
	Let $x \in X$.	Assume that  $\theta$ is convex and   Lipschitz continuous on $X$ with constant $L$.
	Then	the following statements hold:
	\begin{itemize}
		\item[(i)]  For   any direction $v \in X$, the directional derivative  $	\theta'(x; v)$  exists. 
		\item[(ii)] The map $v \mapsto \theta'( x; v)$ is a Lipschitz continuous   function on $X$ with the   constant $L$.
		\item[(iii)] For all $v_1, v_2 \in X$ and $\alpha \ge 0$,
		\[
		\theta'(x; v_1+v_2) \le \theta'( x; v_1) + \theta'( x; v_2), \quad
		\theta'( x; \alpha v_1) = \alpha \,\theta'( x; v_1).
		\]
		In		particular,   the map $v \mapsto \theta'( x; v)$ is a   convex function on $X$.
	\end{itemize}
\end{theorem}

For the sake of convenience, we   always assume in the following that the directional derivative  $	\theta'(x; v)$  exists for any $x \in A$ and for any  $v\in T_A({x})$.
We now introduce an important concept. For any \(x \in A\),   the unit-sphere minimum directional derivative \(\varphi(x)\) is defined  by
\[
\varphi(x) := \inf\bigl\{ \theta'(x; v) : v \in T_A(x),\; \|v\| = 1 \bigr\}.
\]

\begin{remark}   \label{RqXEQ}   Assume that $\theta$ is   Lipschitz continuous on $X$ with constant $L$.       If $\bar{x} \in S_0$,  then $\theta'(x;v)\ge 0$ for any $v\in T_A(\bar{x})$.  In fact, it is clear that 
	$\theta(\bar{x} )=0$ and $\theta(y)\ge 0$ for all $y\in A$. 
	Take any $v\in T_A(\bar{x})$. Then there exist sequences $t_n\downarrow0$ and $v_n\to v$ such that $x_n:=\bar{x}+t_n v_n\in A$. By Lipschitz continuity of $\theta$,
	\[
	0\le\frac{\theta(x_n)}{t_n}= \frac{\theta(\bar{x}+t_n v_n)-\theta(\bar{x})}{t_n}
	\le \frac{\theta(\bar{x}+t_n v)-\theta(\bar{x})}{t_n}+  L  \,\|v_n-v\|.
	\]
	Letting $n\to\infty$, the first term on the right tends to $\theta'(x;v)$, and the second term tends to $0$. Hence $\theta'(\bar{x};v)\ge 0$.   Therefore,  we obtain that if $\bar{x} \in S_0$,  then  $\varphi(\bar{x}) \geq 0$.
\end{remark}

Next,  we discuss the relationship between the unit-sphere minimum directional derivative \(\varphi(x)\) and the global slope \(m(x)\) introduced in Section~3.

\begin{theorem}   \label{ThyvQ} 
	Let \(x\in A\setminus S_0\).	The following statements hold:
	\begin{itemize}
		\item[(i)]  If $A$ is convex and $\theta$ is convex on $A$, then \(m(x)\le -\varphi(x)\).
		\item[(ii)]  If  $\theta$ is   Lipschitz continuous on $X$ with constant $L$, then \(m(x)\ge -\varphi(x)\).
	\end{itemize}
\end{theorem}

\begin{proof} 	(i).  Take any \(y\in A\) with \(\theta(y)<\theta(x)\). Set \(t=\|y-x\|>0\), \(v=(y-x)/t\).  It is clear that \(\|v\|=1\).
Take any sequence  $t_k \downarrow0$ with \( t_k \in (0,1) \).   By the convexity of \(A\),   we obtain
$	x + t_k (y - x) \in A, $	hence  $y - x \in T_A(x).$
Since \( T_A(x)\) is a cone, we have
$	v  \in  {T}_A(x).  $
Since \(\theta\) is convex, we have
\begin{equation}\label{EqpB1}
	\frac{\theta(x+tv)-\theta(x)}{t}\ge \theta'(x;v).
\end{equation}
Since \(v\in T_A(x)\), \(\|v\|=1\)  and  $y= x+tv$, by the definition of \(\varphi\) and (\ref{EqpB1}),
\[
\frac{\theta(x)-\theta(y)}{\|x-y\|}\le -\theta'(x;v)\le \sup_{\substack{v\in T_A(x)\\\|v\|=1}}(-\theta'(x;v))
=-\inf_{\substack{v\in T_A(x)\\\|v\|=1}}\theta'(x;v)=-\varphi(x).
\]
Taking the supremum over all \(y\in A\) with \(\theta(y)<\theta(x)\) yields \(m(x)\le -\varphi(x)\).

(ii).    Take any \(v\in T_A(x)\) with \(\|v\|=1\). 
Then there exist sequences $t_n  \downarrow 0$ and $v_n \to v$ such that $y_n :=  x + t_n v_n \in A$.
Since  $\theta$ is   Lipschitz continuous on $X$ with constant $L$,  we have 
\[
|\theta(y_n)-\theta(x+t_n v)|  =    |\theta(x+t_n v_n)-\theta(x+t_n v)|\le Lt_n\|v_n-v\|.
\]
Hence,
\begin{equation}\label{EqpB2}
	\frac{\theta(y_n)-\theta(x)}{t_n}
	\leq    \frac{\theta(x+t_n v)-\theta(x)}{t_n}+ L \|v_n-v\|, 
\end{equation}
and so 
\begin{equation}\label{EqpB3}  
	\limsup_{n \to\infty}   \frac{\theta(y_n)-\theta(x)}{t_n}   \leq  \theta'(x;v). 
\end{equation}

If \(\theta'(x;v)<0\), then  it follows from  (\ref{EqpB3})  that for sufficiently large \(n\) we have \(\theta(y_n)<\theta(x)\). This together with  (\ref{EqpB2})    implies that 

\begin{eqnarray*}
	m(x) &\ge& \frac{\theta(x)-\theta(y_n)}{\|y_n-x\|}
	\ge \frac{1}{\|v_n\|} \cdot  \frac{\theta(x)-   \theta(y_n)}{t_n} \\
	&\ge&   \frac{1}{\|v_n\|} \cdot \left( \frac{\theta(x)  -  \theta(x+t_n v)}{t_n}   - L \|v_n-v\| \right) .
\end{eqnarray*}
Letting $n\to\infty$, 
by $   \|v_n \|  \to  \|v\|=1$, we get   $m(x)\ge  -\theta'(x;d).$

If \(\theta'(x;d)\ge0\), then \(  m(x)  \ge 0  \ge -\theta'(x;d) \).

Taking the supremum over all \(d\in T_A(x)\) with \(\|d\|=1\) gives 
$$    m(x)\ge  \sup_{\substack{v\in T_A(x)\\\|v\|=1}}(-\theta'(x;v))
=-\inf_{\substack{v\in T_A(x)\\\|v\|=1}}\theta'(x;v)=-\varphi(x).$$
This completes the proof.    \end{proof}

\begin{theorem}  \label{TVhCJK}  Let $X$   be a Banach space  and $A \subseteq X$ be a nonempty closed   set. Assume that $\theta$ is lower semicontinuous on $A$. Then the following statements hold:
	\begin{itemize}
		\item[(i)]   If  $\theta$ is   Lipschitz continuous on $X$ with constant $L$ and 	$ \sup_{x \in A \setminus S_0} \varphi(x) < 0 $, then $\theta$ has an error bound on $A$.
		\item[(ii)] 	If $A$ is convex,  $\theta$ is convex on $A$ and $\theta$ has an error bound on $A$, then $ \sup_{x \in A \setminus S_0} \varphi(x) < 0 $.
	\end{itemize}
\end{theorem}
\begin{proof}    (i).      For  any   \(x\in A\setminus S_0\),    it follows from  Theorem    \ref{ThyvQ}  (ii)  that  \(m(x)\ge -\varphi(x)\).   This together with  $ \sup_{x \in A \setminus S_0} \varphi(x) < 0 $  implies that 
	$$    \inf_{x \in A \setminus S_0} m (x)   \geq  \inf_{x \in A \setminus S_0}    (- \varphi(x)  )      =     -  \sup_{x \in A \setminus S_0} \varphi(x) > 0  .$$
	By Theorem    \ref{TTWCJK},  we obtain that  $\theta$ has an error bound on $A$.
	
	(ii).    Take any \(x\in A\setminus S_0\).   Due to Theorem    \ref{ThyvQ}  (i), we have  \(m(x)\le -\varphi(x)\).   We conclude from  Theorem    \ref{TTWCJK}  that  $\inf_{x \in A \setminus S_0} m (x) >0$,  and so
	$$
	\sup_{x \in A \setminus S_0} \varphi(x) = -  \inf_{x \in A \setminus S_0}    (- \varphi(x)  )  \leq -  \inf_{x \in A \setminus S_0}  m (x)  < 0.
	$$   This completes the proof. 
\end{proof}

\begin{theorem} \label{TRgfd} 
	Let $X$ be a  finite-dimensional   space and  $\bar{x} \in S_0$.    Assume that $\theta$ is   Lipschitz continuous on $X$ with constant $L$.    Then  $	\varphi (\bar x) > 0$ if and only if  	$\theta$ has a local sharp minimum at   $\bar x\in S_0$,	i.e., there exist constants $\kappa > 0$ and $\delta > 0$ such that
	\[
	\theta(x) \ge \kappa \|x - \bar x\| \qquad \forall x \in A \cap B(\bar x, \delta).
	\]
	Moreover, letting $\beta = \sup \left\{ \kappa > 0 : \exists \delta > 0,\ \theta(x) \ge \kappa \|x - \bar x\|,\ \forall x \in A \cap B(\bar x, \delta) \right\}$, we have $\beta = \varphi (\bar x)$.
\end{theorem}

\begin{proof}    $\Rightarrow$.   Assume $\varphi (\bar x) > 0$.   
	Fix any $\epsilon >0$ with $\epsilon <	\varphi (\bar x)$. We prove that there exists $\delta > 0$ such that
	\begin{equation}\label{DvfB1}
		\theta(x) \ge \left( 	\varphi (\bar x)-\epsilon \right) \|x - \bar x\|, \quad \forall x \in A \cap B(\bar x, \delta).
	\end{equation}

	If (\ref{DvfB1}) is not true, then for each positive integer $m$ there exists $x_m \in A$ with
	\[
	0 < \|x_m - \bar x\| < \frac{1}{m} \quad\text{and}\quad \theta(x_m) < \left( 	\varphi (\bar x) - \epsilon \right) \|x_m - \bar x\|.
	\]
	Set $t_m := \|x_m - \bar x\|$ and $v_m := (x_m - \bar x)/t_m$,  then    $t_m   \to 0  $,     $\|v_m\| = 1$ and
	\begin{equation}\label{DvfB2}
		\frac{\theta(x_m)}{t_m} < 	\varphi (\bar x) -\epsilon , \quad \forall m    \in \mathbb{N}.  
	\end{equation}
	
	Since $X$ is finite-dimensional, the unit sphere is compact, so there exists a convergent subsequence (still denoted $v_m$) with $v_m \to v$ and $\|v\| = 1$. By definition of the tangent cone (as $x_m = \bar x + t_m v_m \in A$, $t_m \to 0$),     we have      $v \in T_A(\bar x)$.
	
	Using the Lipschitz continuity of $\theta$,  we get
	\[
	\theta(x_m) = \theta(\bar x + t_m v_m) \ge \theta(\bar x + t_m v) - L t_m \|v_m - v\|.
	\]
	Dividing by $t_m$,
	\[
	\frac{\theta(x_m)}{t_m} \ge \frac{\theta(\bar x + t_m v)}{t_m} - L \|v_m - v\|.
	\]
	Let $m \to \infty$, the first term on the right tends to $\theta'(\bar x; v)$, and the second term tends to $0$. Taking the limit inferior yields
	\[
	\liminf_{m \to \infty} \frac{\theta(x_m)}{t_m} \ge \theta'(\bar x; v).
	\]
	This together with  (\ref{DvfB2})  implies that 
	\begin{equation}\label{DvfB3}
		\theta'(\bar x; v) \le 	\varphi (\bar x) -\epsilon. 
	\end{equation}
	By the definition of $\varphi (\bar x)$,  we have $\theta'(\bar x; v) \ge 	\varphi (\bar x)$,  which contradicts (\ref{DvfB3}). 
	Hence (\ref{DvfB1}) holds.
	From (\ref{DvfB1}) we have $\beta \ge 	\varphi (\bar x) -\epsilon$. Since $\epsilon>0$ is arbitrary, $\beta \ge 	\varphi (\bar x)$.

	$ \Leftarrow $. 
	Assume that there exist $\kappa > 0$   and  $\delta > 0$ such that
	\begin{equation}\label{DvfB4}
		\theta(x) \ge \kappa \|x - \bar x\|,  \quad  \forall x \in A \cap B(\bar x, \delta).
	\end{equation}
	Take any nonzero direction $v \in T_A(\bar x)$.  
	Then there exist sequences $t_k \downarrow 0$ and $v_k \to v$ such that $x_k := \bar x + t_k v_k \in A$. For sufficiently large $k$, $\|x_k - \bar x\| \le t_k \|v_k\| < \delta$, so  it follows from  (\ref{DvfB4})  that
	\begin{equation}\label{DvfB5}
		\theta(x_k) \ge \kappa \|x_k - \bar x\| = \kappa t_k \|v_k\|.  
	\end{equation}
	By the Lipschitz continuity of $\theta$,  we have
	\[
	|\theta(x_k) - \theta(\bar x + t_k v)| \le L \|t_k v_k - t_k v\| = L t_k \|v_k - v\|,
	\]
	and so
	$ 	\theta(\bar x + t_k v) \ge \theta(x_k) - L t_k \|v_k - v\|. $
	This together with   (\ref{DvfB5})  implies that 
	\[
	\frac{\theta(\bar x + t_k v)}{t_k} \ge \kappa \|v_k\| - L \|v_k - v\|.
	\]
	Let  $k \to \infty$,  the right-hand side tends to $\kappa \|v\|$ (because $\|v_k\| \to \|v\|$, $\|v_k - v\| \to 0$). The limit of the left-hand side is exactly the directional derivative $\theta'(\bar x; v)$. Hence
	\[
	\theta'(\bar x; v) \ge \kappa \|v\|,   \quad  \forall v \in T_A(\bar x).
	\]
	In particular,  $\theta'(\bar x; v) \ge \kappa$ for any  $v \in S$, where $S = \{v \in T_A(\bar x) : \|v\| = 1\}$. Taking the infimum gives
	\[
	\varphi (\bar x) = \inf_{v \in S} \theta'(\bar x; v) \ge \kappa > 0.
	\]
	By definition of $\beta$, we also have $	\varphi (\bar x) \ge \beta$. Hence $\beta = 	\varphi (\bar x)$.  This completes the proof. 
\end{proof}

Let  $\bar{x} \in S_0$.  The optimal local error bound constant is defined by 
\[
\tau^*(\bar{x}) := \inf\{ \tau > 0 : \exists \delta > 0,\; \forall x \in A \cap B(\bar{x}, \delta),\; d(x, S_0) \le \tau \theta(x) \},
\]
(with the convention $\inf \emptyset = +\infty$).

\begin{theorem}  \label{TYXWs}    Let $X$ be  a  finite-dimensional   space and  $\bar{x} \in S_0$.  Assume that $\theta$ is   Lipschitz continuous on $X$ with constant $L$.   
	Then the following hold:
	\begin{enumerate}
		\item[(i)]   If $	\varphi(\bar{x}) > 0$, then $\theta$ has a  local error bound at $\bar x$ and
		\[
		\tau^*(\bar{x}) = \frac{1}{	\varphi(\bar{x})} < +\infty.
		\]
		\item[(ii)] 	If $	\varphi(\bar{x}) = 0$ and there exists a direction $v_0\in T_A(\bar{x})\setminus T_{S_0}(\bar{x})$ with $\theta'(\bar{x};v_0)=0$, then $\theta$ does \emph{not} admit a local error bound at $\bar{x}$, i.e., $\tau^*(\bar{x}) = +\infty$.
	\end{enumerate}
\end{theorem}

\begin{proof}   (i).  Assume that $	\varphi(\bar{x}) > 0$.    Fix $\eta \in (0, 	\varphi(\bar{x}))$.  It follows from  (\ref{DvfB1})  that    there exists $\delta_0 > 0$ such that
	\begin{equation}\label{Eyhub1}
		\theta(x) \ge (\varphi(\bar{x}) - \eta) \|x - \bar{x}\|,  \quad	\forall x \in A \cap B(\bar{x}, \delta_0). 
	\end{equation}
	Due to $\bar{x} \in S_0$ and (\ref{Eyhub1}), we have 
	$$ \theta(x) \ge (\varphi(\bar{x}) - \eta) \|x - \bar{x}\|  \ge  (\varphi(\bar{x}) - \eta) d(x, S_0) ,  \quad	\forall x \in A \cap B(\bar{x}, \delta_0), $$
	and so 
	$$ \frac{1}{	\varphi(\bar{x}) - \eta} \, \theta(x)    \ge   d(x, S_0) ,  \quad	\forall x \in A \cap B(\bar{x}, \delta_0). $$
	This yields  that $\theta$ has a  local error bound at $\bar x$ and 
	\[
	\tau^*(\bar{x}) \le \frac{1}{	\varphi(\bar{x}) - \eta}.
	\]
	As $\eta \in (0, 	\varphi(\bar{x}))$ is arbitrary, letting $\eta \downarrow 0$ gives
	\begin{equation}\label{Eyhub2}
		\tau^*(\bar{x}) \le \frac{1}{	\varphi(\bar{x})}.  
	\end{equation}
	
	We show that 
	\begin{equation}\label{Eyhub3}
		S_0 \cap B(\bar x, \delta_0) = \{\bar x\}.
	\end{equation}
	Indeed, if there exists $x' \in S_0 \cap B(\bar x, \delta_0)$ with $x' \neq \bar x$, then by  (\ref{Eyhub1}),
	\[
	0 = \theta(x') \ge  (	\varphi(\bar{x}) - \eta) \|x' - \bar x\| > 0,
	\]
	which is  a contradiction. Therefore,  (\ref{Eyhub3}) holds.
	
	Set $\delta' = \delta_0 /2$. Take any $x \in B(\bar{x}, \delta')$. For any $y \in S_0 \setminus \{\bar{x}\}$, we conclude from  (\ref{Eyhub3})  that $\|y - \bar{x}\| \ge \delta_0$,  and  so
	\[
	\|x - y\| \ge \|y - \bar{x}\| - \|x - \bar{x}\| \ge \delta_0 - \delta' = \delta' > \|x - \bar{x}\|.
	\]
	This means that
	\begin{equation}\label{Eyhub4}
		d(x, S_0) = \|x - \bar{x}\|,   \quad	\forall  x \in A \cap B(\bar{x}, \delta').
	\end{equation}
	
	Suppose $\tau^*(\bar{x}) < 1/	\varphi(\bar{x})$,  then there exists $\tau'$ with $\tau^*(\bar{x}) \le \tau' < 1/	\varphi(\bar{x})$ and $\delta'' > 0$ such that
	\begin{equation}\label{Eyhub5}
		d(x, S_0) \le \tau' \theta(x),   \quad  \forall x \in A \cap B(\bar{x}, \delta'').
	\end{equation}
	Take $\delta = \min\{\delta', \delta ''\}$.  
	It follows from  (\ref{Eyhub4}) and (\ref{Eyhub5})   that
	\begin{equation}\label{Eyhub6}
		\theta(x) \ge \frac{1}{\tau'} \|x - \bar{x}\|,   \quad	\forall  x \in A \cap B(\bar{x}, \delta).
	\end{equation}
	Since $X$ is finite-dimensional, we obtain that  $S := \{v \in T_A(\bar{x}) : \|v\| = 1\}$ is compact. By the Lipschitz continuity of $\theta$,  one can easily verify that   the map $v \mapsto \theta'( x; v)$ is   continuous.  Then there exists  $v_0 \in S$  such that $\theta'(\bar{x}; v_0) = 	\varphi(\bar{x})$.  Due to $v_0 \in  T_A(\bar{x}) $, 
	there exist sequences $t_k \downarrow 0$, $v_k \to v_0$ such that $x_k = \bar{x} + t_k v_k \in A$. For sufficiently large $k$, $x_k \in B(\bar{x}, \delta)$, and applying (\ref{Eyhub6})  yields $\theta(x_k) \ge \frac{1}{\tau'} \|x_k - \bar{x}\| = \frac{1}{\tau'} t_k \|v_k\|$, and so
	\[
	\frac{\theta(x_k)}{t_k} \ge \frac{\|v_k\|}{\tau'}.
	\]
	Since $\|v_k\| \to 1$, taking the limit inferior gives 
	\begin{equation}\label{Eyhub7}
		\liminf_{k \to \infty} \frac{\theta(x_k)}{t_k} \ge \frac{1}{\tau'}.
	\end{equation}
	On the other hand, by the Lipschitz continuity of $\theta$,  we have
	\[
	\frac{\theta(x_k)}{t_k} \le \frac{\theta(\bar{x} + t_k v_0)}{t_k} + L  \|v_k - v_0\|,
	\]
	and taking the limit superior, using the definition of directional derivative and $\|v_k - v_0\| \to 0$, gives
	\[
	\limsup_{k \to \infty} \frac{\theta(x_k)}{t_k} \le \theta'(\bar{x}; v_0) = 	\varphi(\bar{x}).
	\]
	Combining this with  (\ref{Eyhub7}), we get
	\[
	\frac{1}{\tau'} \le 	\varphi(\bar{x}) \quad \Longrightarrow \quad \tau' \ge \frac{1}{	\varphi(\bar{x})},
	\]
	contradicting $\tau' < 1/	\varphi(\bar{x})$.  Therefore, we must have $\tau^*(\bar{x}) \ge 1/	\varphi(\bar{x})$. Together with (\ref{Eyhub2}) we obtain $\tau^*(\bar{x}) = 1/	\varphi(\bar{x})$.

	(ii).    Suppose that $	\varphi(\bar{x}) = 0$ and there exists a direction $v_0\in T_A(\bar{x})\setminus T_{S_0}(\bar{x})$ with $\theta'(\bar{x};v_0)=0$. 
	Since $v_0\notin T_{S_0}(\bar{x})$,    there exist  $\alpha>0$ and $t_0>0$ such that
	\begin{equation}\label{EmkiX1}
		d(\bar{x}+t v_0,\,S_0) \ge \alpha t,    \quad   \forall t\in(0,t_0).
	\end{equation}
	It follows from $v_0\in T_A(\bar{x})$ that    there exist sequences $\{t_k\}\downarrow 0$ and $\{v_k\}\to v_0$ such that    $	x_k := \bar{x}+t_k v_k \in A $ for any $k\in\mathbb{N}$.
	In view of the Lipschitz continuity of $\theta$,  we have
	\[
	\theta(x_k) = \theta(\bar{x}+t_k v_k) \le \theta(\bar{x}+t_k v_0) + L t_k \|v_k - v_0\|,
	\]
	and so
	\[
	\frac{\theta(x_k)}{t_k} \le \frac{\theta(\bar{x}+t_k v_0)}{t_k} + L \|v_k - v_0\|.
	\]
	Let $k\to\infty$.  By   $\theta'(\bar{x};v_0)=0$ and $v_k\to v_0$, we get
	\[
	\limsup_{k\to\infty} \frac{\theta(x_k)}{t_k} \le 0 + 0 = 0.
	\]
	This together with $\theta(x_k) \ge 0$   implies that
	\begin{equation}\label{EmkiX2}
		\lim_{k\to\infty} \frac{\theta(x_k)}{t_k} = 0.
	\end{equation}
	It is clear that
	\begin{equation}\label{EmkiX3}
		d(x_k, S_0) \ge d(\bar{x}+t_k v_0, S_0) - \|x_k - (\bar{x}+t_k v_0)\| = d(\bar{x}+t_k v_0, S_0) - t_k \|v_k - v_0\|.
	\end{equation}
	Take $k$ large enough so that $t_k < t_0$ and $\|v_k - v_0\| \le \alpha/2$.    It follows from  (\ref{EmkiX1})  that  $d(\bar{x}+t_k v_0, S_0) \ge \alpha t_k$.  Combining this with  (\ref{EmkiX3}), we get
	\begin{equation}\label{EmkiX4}
		d(x_k, S_0) \ge \alpha t_k - t_k\frac{\alpha}{2} = \frac{\alpha}{2} t_k > 0.  
	\end{equation}
	
	Suppose that $\theta$ has a  local error bound at $\bar x$.
	Then  there exist $\tau>0$ and $\bar{\delta } >0$ such that
	\begin{equation}\label{EmkiX5}
		d(x,S_0)\le \tau\theta(x),   \quad  \forall x\in A\cap B(\bar{x},\bar{\delta }).
	\end{equation}
	Choose $k$ large enough so that $x_k\in B(\bar{x},\bar{\delta })$.   We conclude from  (\ref{EmkiX4})  and  (\ref{EmkiX5})  that
	\[
	\frac{\alpha}{2} t_k \le d(x_k, S_0) \le \tau \theta(x_k) .
	\]
	This together with    (\ref{EmkiX2})    yields   that 
	$$ 0 < {\alpha  \over 2 \tau} \le {{   \theta \left( {{x_k}} \right)} \over {{t_k}}} \to 0,$$
	which leads to a contradiction.  Therefore, $\theta$ does not have a local error bound at $\bar{x}$, i.e., $\tau^*(\bar{x}) = +\infty$.  This completes the proof. 
\end{proof}

Theorem  \ref{TYXWs} (ii) shows: if $	\varphi(\bar{x})=0$ and, additionally, there exists a direction $v_0\in T_A(\bar{x})\setminus T_{S_0}(\bar{x})$ with $\theta'(\bar{x};v_0)=0$, then a local error bound certainly fails.   The following counterexample demonstrates that the additional condition is essential.

\begin{example}
	Let $X=\mathbb{R}^2$ be equipped with the Euclidean norm.
	Set $A=[0,1]^2$ and define $\theta:X\to\mathbb{R}$ by
	\[
	\theta(x)=\min\{x_1,\,x_2\}, \quad  \forall x = (x_1,x_2)  \in \mathbb{R}^2.
	\]
	Then $\theta$ is concave and Lipschitz continuous on $X$, and on $A$ we have $\theta\ge0$.
	The zero set is
	\[
	S_0=\{x\in A:\theta(x)=0\}
	=\{(x_1,x_2)\in[0,1]^2 : x_1=0\ \text{or}\ x_2=0\}.
	\]
	Take $\bar{x}=(0,0)\in S_0$.
	Clearly $T_A(\bar{x})=\mathbb{R}^2_+$.
	For any direction $v=(v_1,v_2)\in T_A(\bar{x})$,
	\[
	\theta(\bar{x}+tv)=\theta(tv_1,tv_2)=t\min\{v_1,v_2\},
	\]
	hence $\theta'(\bar{x};v)=\min\{v_1,v_2\}$.
	Consequently,
	\[
	\varphi(\bar{x})=\inf_{\substack{v\in T_A(\bar{x})\\\|v\|=1}}
	\theta'(\bar{x};v)
	=\inf_{\substack{v_1,v_2\ge0\\ v_1^2+v_2^2=1}}
	\min\{v_1,v_2\}=0.
	\]
	For any $x=(x_1,x_2)\in A$, the distance to the solution set is $d(x,S_0) = \min\{x_1,\,x_2\}$, which coincides exactly with $\theta(x)$.  
	Hence $d(x,S_0) = \theta(x)$ for all $x\in A$, yielding a local (in fact global) error bound at $\bar{x}$ with constant $\tau=1$.  
	The optimal local error bound constant is $\tau^*(\bar{x}) = 1 < +\infty$.
	
	A straightforward computation shows that
	\begin{equation}\label{Ecyta1}
		T_{S_0}(\bar{x}) = \{(v_1,v_2) \in \mathbb{R}^2: v_1\ge0,\; v_2=0\} \cup \{(v_1,v_2)  \in \mathbb{R}^2: v_1=0,\; v_2\ge0\}.
	\end{equation}
	
	Finally, we verify that the condition ``there exists a direction $v_0\in T_A(\bar{x})\setminus T_{S_0}(\bar{x})$ with $\theta'(\bar{x};v_0)=0$'' is not satisfied.  
	Suppose, to the contrary, that such a direction $\bar{v}=(v_1,v_2)\in T_A(\bar{x})\setminus T_{S_0}(\bar{x})$ exists with $\theta'(\bar{x};\bar{v})=0$.  
	Since $\bar{v}\in T_A(\bar{x})$, we have $v_1\ge0$ and $v_2\ge0$.  From $\theta'(\bar{x};\bar{v}) = \min\{v_1,v_2\}=0$, it follows that $v_1=0$ or $v_2=0$.  By \eqref{Ecyta1}, this forces $\bar{v}\in T_{S_0}(\bar{x})$, contradicting the choice of $\bar{v}$.  
	Thus no such direction exists.
	
	Thus, the hypothesis of Theorem~\ref{TYXWs}(ii) is not satisfied, yet a local error bound still exists. This demonstrates that the additional condition in Theorem~\ref{TYXWs}(ii) cannot be removed.
\end{example}

\begin{corollary}  \label{CtgXEQ}     Let  $\bar{x} \in S_0$. Assume that     $\theta$ is   Lipschitz continuous on $X$ with constant $L$   and     $\theta$ has a  local error bound at $\bar x$. Then 	
	\begin{equation}\label{EcxB1}
		\{v \in T_A(\bar x) : \theta'(\bar x; v) = 0\} \subseteq T_{S_0}(\bar x).
	\end{equation}
\end{corollary}
\begin{proof}   Suppose, to the contrary, that  (\ref{EcxB1}) is false.  Then there exists \(v_0\in T_A(\bar x)\setminus T_{S_0}(\bar x)\) such that \(\theta'(\bar x;v_0)=0\).   It is clear that $v_0 \neq 0$ and  \(\varphi(\bar x)=0\).  By Theorem   \ref{TYXWs}   (ii)   (in the proof of Theorem   \ref{TYXWs}   (ii), it is not necessary to assume that $X$ is a finite-dimensional space), we get that the local error bound cannot hold, contradicting the hypothesis.  Hence  (\ref{EcxB1}) must be true.   This completes the proof. 
\end{proof}

\begin{theorem}  \label{TnbEQ}     Let \(X\) be a  finite-dimensional  space, \(A\subseteq X\) be a  nonempty closed set and 
	\(\bar{x}\in S_0\).  Assume that \(\theta\) is    Lipschitz continuous on \(X\) with constant $L$ and \(\theta\) has a  local error bound  at \(\bar{x}\), i.e., there exist   $\tau>0$  and  $ \delta>0$ such that
	\begin{equation}\label{EfdB1}
		d (x,S_0)\le \tau\,\theta(x),   \quad\forall x\in A\cap B(\bar{x},\delta).
	\end{equation}
	Then 
	\[
	\theta'(\bar{x};v)\;\ge\;\frac{1}{\tau}\,d \bigl(v,\;T_{S_0}(\bar{x})\bigr),   \quad  \forall v\in T_A(\bar{x}).
	\]
\end{theorem}

\begin{proof}
	Take any \(v\in T_A(\bar{x})\). By the definition of the tangent cone, there exist a positive sequence \(\{t_k\}\downarrow 0\) and a sequence \(\{v_k\}\subset X\) such that
	\[
	v_k\to v,\qquad x_k:=\bar{x}+t_k v_k\in A,\ \  \forall k \in \mathbb{N}.
	\]
	Due to  \(x_k\to \bar{x}\), we can choose \(k_0 \in  \mathbb{N}\) large enough so that   \(x_k\in B(\bar{x},\delta)\) for all \(k\ge k_0\).   It follows from  (\ref{EfdB1})  that
	\begin{equation}\label{EfdB2}
		d (x_k,S_0)\le \tau\,\theta(x_k), \quad  \forall k\ge k_0 .  
	\end{equation}
	Fix an arbitrary \(\varepsilon>0\). For each \(k\ge k_0\), by the definition of distance, we can pick \(y_k\in S_0\) satisfying
	\begin{equation}\label{EfdB3}
		\|x_k-y_k\| \le d (x_k,S_0) + \varepsilon t_k.  
	\end{equation}
	Set
	\[
	w_k:=\frac{y_k-\bar{x}}{t_k}\;\in X .
	\]
	From (\ref{EfdB2}) and (\ref{EfdB3}) we immediately obtain
	\begin{equation}\label{EfdB4}
		\|x_k-y_k\| \le \tau\,\theta(x_k) + \varepsilon t_k.  
	\end{equation}
	We now show that \(\{w_k\}\) is bounded. Indeed, using (\ref{EfdB4})  and \(x_k=\bar{x}+t_k v_k\),  we get
	\begin{equation}\label{EfdB5}
		\|w_k\| = \frac{\|y_k-\bar{x}\|}{t_k}
		\le \frac{\|y_k-x_k\|}{t_k} + \frac{\|x_k-\bar{x}\|}{t_k}
		\le \frac{\tau\,\theta(x_k)+\varepsilon t_k}{t_k} + \|v_k\|.
	\end{equation}
	By the Lipschitz continuity of $\theta$,  we have
	\[
	\frac{\theta(x_k)}{t_k} \le \frac{\theta(\bar{x}+t_k v_0)}{t_k} + L \|v_k - v_0\|,
	\]
	and so 
	\begin{equation}\label{EfdB6}	
		\limsup_{k\to\infty} \frac{\theta(x_k)}{t_k} \le \theta'(\bar{x};v_0).	
	\end{equation}
	Hence the sequence \(\{\theta(x_k)/t_k\}\) is bounded. Moreover, \(v_k\to v\) implies that \(\{\|v_k\|\}\) is bounded. This together with   (\ref{EfdB5})  implies that   \(\{w_k\}\) is bounded.
	Because \(X\) is finite-dimensional,   there exist a subsequence \(\{k_j\}\) and   \(w\in X\) such that $ w_{k_j} \to w $.
	
	By construction, \(y_{k_j} = \bar{x} + t_{k_j} w_{k_j} \in S_0\) and \(t_{k_j}\downarrow 0\). According to the definition of the tangent cone, this implies	$		w \in T_{S_0}(\bar{x}).  $
	Now rewrite (\ref{EfdB4}) and divide by \(t_k\),
	\[
	\left\|v_k - w_k\right\| = \frac{\|x_k-y_k\|}{t_k}
	\le \tau\,\frac{\theta(x_k)}{t_k} + \varepsilon,  \quad  \forall    k\ge k_0.
	\]
	Taking the limit along the subsequence \(j\to\infty\) and using \(v_{k_j}\to v\), \(w_{k_j}\to w\), and  (\ref{EfdB6}), we get
	\[
	\|v - w\| \le \tau\,\theta'(\bar{x};v) + \varepsilon .
	\]
	Since \(w\in T_{S_0}(\bar{x})\), it follows that
	\[
	d \bigl(v,\,T_{S_0}(\bar{x})\bigr) \le \|v-w\| \le \tau\,\theta'(\bar{x};v) + \varepsilon .
	\]
	By the arbitrariness of   \(\varepsilon>0\),  we conclude that
	$	d  \bigl(v,\,T_{S_0}(\bar{x})\bigr) \le \tau\,\theta'(\bar{x};v),	$
	i.e.,
	\[
	\theta'(\bar{x};v) \ge \frac{1}{\tau}\,d \bigl(v,\,T_{S_0}(\bar{x})\bigr).
	\]
	This completes the proof.  \end{proof}

 \bigskip
 \noindent
 {\bf Disclosure statement}  \\
No potential conflict of interest was reported by the author(s).
 
\bigskip
\noindent
{\bf Acknowledgements}  \\
This work was supported by the National Natural Science Foundation of China [11801257] and the Natural Science Foundation of Jiangxi Province [20232BAB211012].

\end{document}